\documentclass[10pt]{article}
\usepackage{amsthm,amsmath,amsfonts}
\usepackage{fullpage,enumerate}
\usepackage{tikz}
\usepackage[colorlinks,citecolor=blue,urlcolor=blue]{hyperref}
\usepackage{hypernat}

\newcommand{\f}{\phi}
\newcommand{\cE}{\mathcal{E}}
\newcommand{\cF}{\mathcal{F}}
\newcommand{\cU}{\mathcal{U}}

\newcommand{\To}{\Rightarrow}
\newcommand{\vep}{\varepsilon}

\newtheorem{thm}{Theorem}
\newtheorem{lem}{Lemma}

\newtheorem{rem}{Remark}

\begin{document}

\date{}

\numberwithin{equation}{section}

\title{A stochastic model for the evolution of the influenza virus}

\bigskip
\author{J. Theodore Cox\thanks{ Supported in part by an NSF grant.}
\\ {\it Syracuse University} 
 \and
  Rinaldo B. Schinazi\thanks{ Supported in part by an NSA grant.}\\ {\it University of Colorado, Colorado Springs}}
\maketitle

{\bf Abstract.} Consider a birth and death chain to model
the number of types of a given virus. Each type gives birth
to a new type at rate $\lambda$ and dies at rate 1. Each
type is also assigned a fitness. When a death occurs either
the least fit type dies (with probability $1-r$) or we kill
a type at random (with probability $r$). We show that this
random killing has a large effect (for any $r>0$) on the
behavior of the model when $\lambda<1$. The
behavior of the model with $r>0$ and $\lambda<1$ is
consistent with features of the phylogenetic tree of
influenza.

\bigskip

{\bf Key words:} phylogenetic tree, influenza, stochastic model, mutation 

{\bf AMS Classification:} 60K35

\section{Introduction.}

Consider the following model for the evolution of a virus. 
The model depends on two parameters, $\lambda>0$ and $r\in[0,1]$. We think of $\lambda$ as the mutation rate. 
The number of types at time $t$ is denoted by $X(t)$, 
a birth-death process which makes transitions
\[
n \to \begin{cases}
n+1 & \text{at rate }n\lambda \text{ for }n\ge 1\\
n-1 & \text{at rate }n \text{ for }n\ge 2
\end{cases}
\]
(the number of types is never less than one).  Each virus
type has a fitness $\f$, chosen at random from the uniform
(0,1) distribution when it is created (so each new type is
different from all previous types).  When a type dies the
type that is chosen to die is, with probability $r$,
selected uniformly among the existing types, and with
probability $1-r$ the type with minimal fitness. We will say
that with probability $r$ a random killing occurs.

The model with $r=0$ (the least fit type type is always
killed) was introduced by Liggett and Schinazi in \cite{LS09}.
Several articles have since been written on closely
related models, see \cite{BMR}, \cite{GMS1}, \cite{GMS2} and
\cite{MV}.  ``Kill the least fit'' models go back to at
least \cite{BK}. The model with random killing (i.e. $r>0$) is a natural
extension for at least two reasons. From a modeling perspective ``Kill
the least fit'' is quite natural. However, assuming that this is always the case
is not. Random events should occasionally prevent this
transition from happening.
Furthermore, from a mathematical perspective it seems interesting to study the effect
of small random perturbations of the basic model. As we will see  they can have 
major effects on the behavior of the model.

We are interested in 
\begin{align*}
\f_t&=\f_t^r= \text{ the maximal fitness of the types alive at time
  $t$}, \\
a_t&= a^r_t=\text{ the age of the type with maximal fitness at time $t$}
\end{align*}
(if a type is created at time $s$ then its age at time $t>s$
is $t-s$). We start the process with a single individual. 
We assume that its fitness $\f_0$ is
uniformly distributed on (0,1), and initially we take
$a_0=0$. 

Let $\To$ denote weak convergence and $\to_p$ denote
convergence in probability. The following theorem summarizes
the main results of \cite{LS09}.

\begin{thm}[\cite{LS09}] Assume $r=0$, and $Y$ is uniformly
  distributed on the interval (0,1).
\begin{enumerate}[{(a)}]
\item If $\lambda \le 1$ then $a_t/t \To Y$
as $t\to\infty$. 
\item If $\lambda >1$ then $a_t/t \to_p 0$ as $t\to\infty$.
\end{enumerate}
\end{thm}

When $\lambda<1$, $X(t)$ converges in distribution to its
stationary distribution, and hence at any given time there
will not be many  types. In this case, (a) above shows that 
the fittest type at time $t$ will have been around for order
of time $t$. As noted in \cite{LS09}, this is consistent
with the observed structure of an influenza tree. When
$\lambda>1$, $X(t)$ tends to infinity as $t\to\infty$, and
(b) shows that the fittest type at time $t$ has been around
only for only $o(t)$ time. As noted in \cite{LS09},  this is
consistent with the observed structure of an HIV tree. In
the critical case $\lambda=1$ we have something inbetween
these two pictures. It is easy to see that in all cases the maximal fitness
$\f_t\to 1$ as $t\to\infty$.

Theorem~1 shows that the model with $r=0$ can, by
adjusting $\lambda$, describe rather different
evolutions. Nevertheless, it has some limitations. The
maximal fitness always tends to 1, and for $\lambda\le1$ the age
$a_t$ tends to infinity. As shown below, the model with
random killing ($r>0$) allows for the possibilities that
$\f_t\not\to 1$ and $a_t\not\to\infty$. 

Before proceeding to our results for the $r>0$ case we
resolve one question left open by Theorem~1. Namely, (b)
leaves open the two possibilities: $a_t$ is (stochastically)
bounded as $t\to\infty$, or $a_t\to\infty$. It turns out
that $a_t$ does not tend to infinity, instead it 
converges in distribution. For the sake of
completeness, we include the behavior of the maximal fitness
in the following result.

\begin{thm} Assume $r=0$, and let $\cE$ be a mean one exponential
  random variable.
\begin{enumerate}[(a)]
\item For $\lambda>0$, $\phi_t\to 1$ a.s.\ as
$t\to\infty$. 
\item For $\lambda >1$, $a_t\To \frac1{\lambda-1}\cE$ as
$t\to\infty$.
\end{enumerate}
\end{thm}

We turn to the case of random killings ($r>0$) and focus on
the $\lambda<1$ case. We see that the behaviors of the
maximal fitness and age processes are quite different 
from the $r=0$ case.

\begin{thm} Assume $r>0$ and $\lambda<1$. Then 
\begin{enumerate}[(a)]
\item $\f_t$ converges in distribution as
  $t\to\infty$ to a nondegenerate limit law, and
\item $a_t$ converges in distribution as
  $t\to\infty$ to a nondegenerate limit law.
\end{enumerate}
\end{thm}

Theorem 3 is consistent with features of the influenza phylogenetic
tree.  The most fit type lasts a finite random time and then
is replaced by a new most fit type and so on. As desired
$a_t$ does not go to infinity with $t$ and $\f_t$ does not
go to 1. Instead they converge to nondegenerate limits. 

Turning to the $r>0,\lambda>1$ case, our results are
less complete. We can show that the fitness $\phi_t$ tends to one
as $t\to\infty$, but we cannot show, as we conjecture, that the age
$a_t$ does not tend to infinity.

\begin{thm}
For $r>0$ and $\lambda\ge1$, $\phi_t\to_p1$ as $t\to\infty$.
\end{thm}


In the next section we give the proof of Theorem~2. In
Section 3 we give a construction that we use to prove
Theorem~3. The construction allows us to write down a
renewal type description of the limit laws for both the
fitness and age processes. In Section 4 we use a different
construction to prove Theorem~4. 

\section{Proof of Theorem~2}
Let us dispense with the easy convergence $\phi_t\to 1$.
Let $B(t)$ be the number of types
created by time $t$, and let $\cU_1,\cU_2,\dots$ be the successive
iid uniform $(0,1)$ random variables created as the process evolves.
Then $\f_t=\max\{\cU_1,\dots,\cU _{B(t)}\}$. It is easy to
see that $\max\{\cU_1,\dots,\cU_n\}\to 1$ a.s.\
as $n\to\infty$. Since $B(t)\to\infty$ a.s.\ we get $\f_t\to 1$ a.s.

For (b), fix $\lambda>1$ and  recall the notation of Section~3 of
\cite{LS09}. Following the notation there, let $T_n$ be the
first time $X_t$ reaches $n$, let $N(t)=\sup\{n: T_n\le
t\}$, and set
\[
\zeta(n)=T_n - \dfrac{\log n}{\lambda-1} .
\]
We need an
improvement of Lemma~1 of \cite{LS09}. 

\begin{lem} With probability one,
  $\lim_{t\to\infty}N(t)e^{-(\lambda-1)t}
= e^{-(\lambda -1)\zeta_\infty}$, a strictly positive finite
limit.
\end{lem}
\begin{proof}
It was shown at the end of the proof of Lemma~1 in \cite{LS09}
that $ \zeta(n)\to \zeta_\infty\ a.s. \text{ as }n\to\infty$
for some finite random variable $\zeta_\infty$.   
Since $N(t)\to\infty$ as $t\to\infty$ we also have 
$\zeta(N(t))\to \zeta_\infty$ a.s.\ as $t\to\infty$.
By definition, 
\begin{equation}\label{TN}
T_{N(t)}\le t <T_{N(t)+1}
\end{equation}
so
\[
\zeta_N(t) \le t - \dfrac{\log(N(t))}{\lambda-1}
\]
or
\[
\log(N(t)) - (\lambda-1)t \le -(\lambda-1)\zeta(N(t)) .
\]
This implies 
$N(t)e^{-(\lambda-1)t} \le e^{-(\lambda-1)\zeta(N(t))}$
and therefore
\[
\limsup_{t\to\infty}N(t)e^{-(\lambda-1)t} \le
e^{-(\lambda-1)\zeta_\infty}
\ a.s.
\]

To get an inequality in the reverse direction we note that \eqref{TN} implies
\[
t - \dfrac{\log(N(t)+1)}{\lambda-1} <\zeta(N(t)+1),
\]
or 
\[
\log(N(t)+1)-(\lambda-1)t > -(\lambda-1)\zeta(N(t)+1) .
\]
This implies 
$
(N(t)+1)e^{-(\lambda-1)t}>e^{-(\lambda-1)\zeta(N(t)+1)}
$
and therefore 
\[
\liminf_{t\to\infty}N(t)e^{-(\lambda-1)t} \ge
e^{-(\lambda-1)\zeta_\infty}
\ a.s.
\]
This completes the proof, since $\zeta_\infty$ is positive
and finite with probability one.
\end{proof}

When $r=0$ the maximal fitness $\phi_t$ is increasing in
$t$. This implies that for $s<t$, $a_t\ge t-s$ if and only if 
$\phi_s=\phi_t$. Let $S_n$ be the number of
  types produced up to time $T_n$.  By (1) and (2) in \cite{LS09}, 
\begin{equation}\label{LS12}
E\Big[ \dfrac{S_{N(s)}}{S_{N(t)+1}}, N(s)<N(t)
\Big] \le 
P(\phi_s = \phi_t , N(s)<N(t) ) \le 
E\Big[ \dfrac{S_{N(s)+1}}{S_{N(t)}}, N(s)<N(t)
\Big]. 
\end{equation}
Fix $u>0$ and let $s=t-u$. By Lemma~4,
$P(N(s)<N(t))\to 1$ as $t\to\infty$, so it suffices to prove
that both the left-side and right-side of \eqref{LS12}
converge to $e^{-(\lambda-1)u}$.

It was shown in
\cite{LS09} that $S_n/n$ converges a.s.\ to a finite
positive limit as $n\to\infty$. By this fact,
$N(t)\to\infty$, and Lemma~1,
\[
\dfrac{S_{N(s)+1}}{S_{N(t)}} = 
\dfrac{S_{N(s)+1}}{N(s)+1}
\dfrac{N(t)}{S_{N(t)}}
\dfrac{N(s)+1}{N(t)}
\to
e^{-(\lambda-1)u} \quad a.s.
\]
It follows that the right-side of \eqref{LS12} converges to
$e^{-(\lambda-1)u}$ as $t\to\infty$. A similar argument
handles the left-side of \eqref{LS12}. This completes the
proof of Theorem~2. 

\section{Proof of Theorem~3.}

Throughout this section $0<r\le 1$ and $0<\lambda<1$ are
fixed. We first extend the notation of Section~2 of \cite{LS09}
making the following definitions and observations. 

\begin{enumerate}[(1)]
\item Put $T_0=0$ and for $n\ge 1$ let
  $T_n$ be the time of the $n$th return of $X(t)$ to state 1. 
The ``interarrival times'' times $\{T_n-T_{n-1},n\ge 1\}$
are iid random variables. .
\item For $n\ge 1$ let $\xi_n$ be the duration of the $n$th
  sojourn time in state   1,
\[
\xi_{n} =
  \inf\{t>T_{n-1}: X_t\ne 1\} .
\]
The random variables $\{\xi_n,n\ge 1\}$
  are iid exponential with parameter $\lambda$. Note also that for $n\geq 0$ 
$\sigma(T_0,\dots,T_n)$ is 
  independent of $\sigma(\xi_{n+1},\xi_{n+2},\dots)$. 
\item For $n\ge 1$ let
  $u_n$ be the uniform random variable created at time
  $T_{n-1}+\xi_n$, when $X(t)$ jumps from 1 to 2. 
  At time $T_{n-1}+\xi_n$ there are two types, with
  fitnesses $\f(T_{n-1}), u_n$. 
  The $\{u_n,n\ge 1\}$ are iid uniform (0,1) rv's, independent of 
  the sequences $\{T_n,n\ge 0\}$ and $\{\xi_n,n\ge 1\}$.
\item For $n\ge 1$ let $\eta_{n}$ be the duration of the sojourn time 
in 2 starting at time $T_{n-1}+\xi_n$,
\[
\eta_n=\inf\{t>T_{n-1}+\xi_n: X_t\ne 2\}.
\] 
The random variables $\{\eta_n,n\ge 1\}$ are iid exponential
with parameter $2\lambda+2$, 
independent of $\{\xi_n,n\ge 1\}$ and $\{u_n,n\ge
1\}$. Furthermore, $\sigma(T_0,\dots,T_n)$ is 
  independent of $\sigma(\eta_{n+1},\eta_{n+2},\dots)$. 
\item For $n\ge 1$ let $T'_n=T_{n-1}+\xi_n+\eta_n$. 
 For all $t\in[T_{n-1}+\xi_n,T'_n)$   here are exactly two types,
  the fitnesses are $\f(T_{n-1}), u_n$. 
\item At time $T'_{n}-$, if $X(t)$ jumps to 1, with
  probability $r$ one of the types $u_n,\f(T_{n-1})$ 
  is chosen to be killed. For $n\ge 1$ let 
\[
\vep_{n}=\begin{cases}
1&\text{at time $T'_n-$, $X_t$ jumps to 1 and the type $\f(T_{n-1})$
  is killed by random killing}\\
0& \text{otherwise.}
\end{cases}
\]
Note that we do not include in the event $\{\vep_n=1\}$ the
possibility that $\f(T_{n-1})<u_n$ and the least fit type is
killed with probability $1-r$. 
The random variables $\{\vep_n,n\ge 1\}$ are iid Bernoulli
with mean 
\[
p=\frac{2}{2(1+\lambda)}\frac{r}2=\frac{r}{2(1+\lambda)}>0
\]
Also, the sequence $\{\vep_n,n\ge 1\}$
is independent of the sequence $\{u_n,n\ge 1\}$, and
$\sigma(T_k,\xi_k,\eta_k, k\le n)$ is independent of
$\sigma(\vep_{n+1},\vep_{n+2},\dots)$. 
\item To consider the return times $T_j$ corresponding to
  the event $\{\vep_n=1\}$, put $\kappa_0=0$,
  $R_0=0$, and for $n\ge 1$ define 
\[
\kappa_{n}= \inf\{k>\kappa_{n-1} : \vep_k=1\} \text{ and }
R_n=T_{\kappa_n} .
\]
The random variables $\{R_n-R_{n-1},n\ge 1\}$ are iid, with $\mu=ER_1\in(0,\infty]$
and at the times $R_n$, $n\ge 1$,
\begin{equation}\label{regenfitage}
\begin{aligned}
\f_{R_n}&=u_{\kappa_{n}} \text{ is uniform on }(0,1)\\
a_{R_n}&= \eta_{\kappa_n} \text{ is exponential with parameter }2(\lambda+1).
\end{aligned}
\end{equation}
\end{enumerate}
The construction is illustrated in Figure~1 below, in which
$\vep_1=0$, $\vep_2=1$ and $R_1=T_2$.

\bigskip
\begin{center}
\begin{tikzpicture}[scale=.75]\large
\draw[dotted,->] (0,0) -- (11,0);
\draw[dotted,->] (0,0) -- (0,5);
\draw (0,0) node[left] {$1$};
\draw (0,2) node[left] {$2$};
\draw (0,4) node[left] {$3$};
\draw (11,0) node[right] {$t$};
\draw[very thick] (0,0) -- (2,0);
\draw (0,-.1) -- (0,.1);
\draw (0,5.0) node[above] {$X(t)$};
\draw (0,-.5) node[below] {$T_0$};
\draw (0,-1.5) node[below] { $R_0$};
\draw (1,0) node[above] { $\xi_1$};
\draw[very thick] (2,2) -- (3,2);
\draw (2,-.1) -- (2,.1);
\draw (2.5,2) node[above] { $\eta_1$};
\draw[very thick] (3,4) -- (4.5,4);
\draw (3,-.1) -- (3,.1);
\draw[very thick] (4.5,2) -- (6.5,2);
\draw (4.5,-.1) -- (4.5,.1);
\draw[very thick] (6.5,0) -- (7.5,0);
\draw (6.5,-.1) -- (6.5,.1);
\draw (6.5,-.5) node[below] { $T_1$};
\draw (7,0) node[above] { $\xi_2$};
\draw[very thick] (7.5,2) -- (9,2);
\draw (7.5,-.1) -- (7.5,.1);
\draw (8.25,2) node[above] { $\eta_2$};
\draw[very thick] (9,0) -- (10,0);
\draw (9,-.1) -- (9,.1);
\draw (9,-.5) node[below] { $T_2$};
\draw (9,-1.5) node[below] { $R_1$};
\draw (9.5,0) node[above] { $\xi_3$};
\end{tikzpicture}\\
Figure 1
\end{center}

By \eqref{regenfitage}, at time $R_1$ there is a single
type, its fitness has the uniform distribution
on $(0,1)$, and its age has the exponential distribution with
parameter $2(\lambda+1)$. Furthermore, given this
information, the distribution of our process for $t\ge R_1$
is independent of what has happened before time $R_1$. It
follows that if we start at time $0$ with a single type
with fitness uniformly distributed on $(0,1)$ and age
exponentially distributed with parameter $2(\lambda+1)$ then
$R_1$ is a regeneration time. The strong 
Markov property now implies the following result.

\begin{lem} If $\f_0$ is uniformly distributed on
  $(0,1)$ and $a_0$ is exponentially distributed with parameter
  $2(\lambda+2)$ then for $t>0$, 
\begin{equation}\label{fitrenew}
P(\f_t\le v, R_1\le t) = \int_0^tP(R_1\in
  ds)P(\f_{t-s}\le v), \quad 0<v<1, 
\end{equation}
and
\begin{equation}\label{agerenew}
P(a_t\le x, R_1\le t) = \int_0^tP(R_1\in
  ds)P(a_{t-s}\le x), \quad x>0.
\end{equation}
\end{lem}

\begin{rem} The fitness process does not depend on
  the age process, so \eqref{fitrenew} holds regardless of
  the distribution of $a_0$.
\end{rem}

In order to make use of \eqref{fitrenew} and \eqref{agerenew}
we will need information on the tail of the distribution of
$R_1$, which is provided by our next result. 

\begin{lem} For $\lambda<1$ there are constants $C,\gamma$ such that
  $P(R_1>t)\le Ce^{-\gamma t}$. In particular, $E(R_1)<\infty$.
\end{lem}
\begin{proof} We are going to use Gronwall's
  inequality. Let $\tilde X(t)$ denote $X(t)$ starting at 3 instead of 1, 
  let $\tilde T_1$ be the first time $\tilde X(t)$ reaches
  1, and let $\tilde R_1$ be
  defined analogously to $R_1$. By a simple coupling it is
  clear that $P(R_1>t)\le P(\tilde R_1>t)$ for all $t>0$.
Let $\tau$ be the first time $\tilde X(t)$
  reaches 2 after reaching 0,
\[
\tau = \inf\{t>\tilde T_1 : \tilde X(t)=2\},
\]
and let $\tilde \eta$ be an independent exponential random
variable with parameter $2\lambda+2$. Finally, let 
$\tau'=\tau+\tilde\eta$. 
By the Markov property, 
\[
P(\tilde R_1>t) = P(\tau'>t) +
(1-p)\int_0^tP(\tau'\in ds) P(\tilde
R_1>t-s) .
\]
It follows now from Gronwall's inequality that
\[
P(\tilde R_1>t) \le P(\tau'>t)e^{(1-p)P(\tau'\le t)}
\le e P(\tilde R_1>t).
\]
Since $\tau'=\tau+\tilde\eta$, it suffices now to prove that
$\tau$ has an exponential tail.

For the remainder of this argument we amend the dynamics of
$X(t)$ to include a transition from 1 to 0 at rate
$1$, and treat 0 as a trap. If we let $\tau_0$ be the first
hitting time of $0$, then $\tau_0>\tau$, so the final
reduction is to prove that for some constants $C,\gamma$,
\[
P(\tau_0>t | X(0)=3) = P(X(t)=0|X(0)=3)\le Ce^{-\gamma t} .
\]
The amended birth-death process $X(t)$ is a continuous
time branching process, as shown in Section~III.5 of \cite{AN},
where an explicit expression for the generating function
$\sum_{k=0}^\infty s^kP(X(t)=k|X(0)=1)$ is given. Setting
$s=0$ we obtain
\[
P(X(t)\ne 0|X(0)=1) = e^{-(1-\lambda
  )t}\dfrac{1-\lambda}{1-\lambda e^{-(1-\lambda)t}} \le 
e^{-(1-\lambda)t} .
\]
By the branching property, we get $P(X(t)\ne 0|X(0)=3) \le
3P(X(t)\ne 0|X(0)=1)$ so we are done. 
\end{proof}

With these facts established we begin the proof of part (a)
of Theorem~3. Let $F(t)=P(R_1\le t)$, and let 
$U=\sum   F^{(*n)}$ be the corresponding renewal
  function, $U(t) = \sum_n P(R_n\le t)$.
Fix $v\in(0,1)$ and define 
\[
h_v(t) = P(\f_t\le v, R_1>t) \text{ and }
H_v(t) = P(\f_t\le v) .
\]
By decomposing the event defining $H_v(t)$ according to the
value of $R_1$,  and using \eqref{fitrenew}, we have 
\begin{equation}\label{renewal1}
H_v(t) = h_v(t) + P(\f_t\le v, R_1\le t)
= h_v(t) + \int_0^t H_v(t-s) F(ds) .
\end{equation}
It follows from Theorem 4.4.4 of \cite{Dur} that the solution
to this renewal equation is given by
\begin{equation}\label{fitnewrep1}
H_v(t) = \int_0^t h_v(t-s)U(ds) .
\end{equation}
We claim that 
\begin{equation}\label{dri}
\text{$h_v(t)$ is directly Riemann integrable if $\lambda<1$.}
\end{equation}
Given this, a standard renewal theorem 
(Theorem 4.4.5 of \cite{Dur}) implies that 
\begin{equation}\label{fitlimit1}
H_v(t) \to \frac1\mu\int_0^\infty h_v(s) ds \text{ as }t\to\infty
\end{equation}
or 
\begin{equation}\label{fitrep}
\lim_{t\to\infty}P(\f_t\le v) =
\frac1\mu\int_0^\infty P(\f_s\le v, R_1>s) \, ds 
\end{equation}
(recall that $\mu=E(R_1)$). 
For $\lambda=1$ we still have \eqref{fitnewrep1}, but
  not \eqref{fitlimit1} since this depends on $\mu<\infty$.

In view of the fact that $P(R_1>t)$ decays exponentially
fast, to prove \eqref{dri} it suffices to prove that
$h_v(t)$ is a continuous function of $t$. For $s<t$ let 
$\Gamma_{s,t}$ be the event that the birth-death process
makes no transitions in the time interval $[s,t]$. On
$\Gamma_{s,t}$, $\f_\cdot$ cannot change, and $R_1>s$ if and only
if $R_1>t$, so that
\[
P(\{\f_s\le v, R_1>s\}\cap \Gamma_{s,t}) =
P(\{\f_t\le v, R_1>t\}\cap \Gamma_{s,t}). 
\]
It follows that 
\begin{align*}
|h_v(s)-h_v(t)| &\le P(\Gamma^c_{s,t}) \\
&= \sum_{k=1}^\infty
P(X(s)=k)(1-e^{-k(\lambda+1)(t-s)})\\
&\le 
\sum_{k=1}^\infty
P(X(s)=k)k(\lambda+1)(t-s)\\
&= (t-s)(\lambda+1)E(X(s)).
\end{align*}
For $\lambda<1$, $\sup_sE(X(s))<\infty$, so we have proved
that $h_v$ is continuous and directly Riemann integrable.

For Theorem~3(a), we suppose first that $a_0$ is exponential with
parameter $2(\lambda+1)$, so that \eqref{agerenew} holds. Now we
follow the previous argument. 
Fix $x>0$ and define 
\[
g_x(t) = P(a_t\le x, R_1>t) \text{ and }
G_x(t) = P(a_t\le x)
\]
As in the argument for Theorem~2(b), for $\lambda\le 1$ we have 
\begin{equation}\label{renewal2}
G_x(t) = g_x(t) + \int_0^t G_x(t-s)
F(ds) = \int_0^t g_x(t-s)U(ds) .
\end{equation}
For $\lambda<1$, an argument similar to the one for $h_v(t)$
shows that $g_x(t)$ is directly
Riemann integrable, and by the renewal theorem 
\begin{equation}\label{agelimit1}
G_x(t) \to \frac1\mu\int_0^\infty g_x(s) ds \text{ as }t\to\infty
\end{equation}
or 
\begin{equation}\label{agerep}
\lim_{t\to\infty}P(a_t\le x) =
\frac1\mu\int_0^\infty P(a_s\le x, R_1>s) \, ds .
\end{equation}

Given any $\tilde a_0\ge 0$, by using the same birth-death
process and sequence of uniform random variables, we may
construct an age process $\tilde a_t$ with the property that
\begin{equation}\label{aa}
\tilde a_t= a_t \text{ if }t\ge R_1 .
\end{equation}
This is because at time $R_1=T_k$ for some $k$, the most fit type is the
uniform random variable created at time $T_k+\xi_{k}$, and
has age $\eta_k=a_{R_1}=\tilde a_{R_1}$. After time $R_1$
the two age processes are identical. By \eqref{aa}, 
$P(a_t\ne\tilde a_t) \to 0$ as $t\to\infty$,
and therefore for any $\tilde a_0$,
\begin{equation}\label{agerep2}
\lim_{t\to\infty}P(\tilde a_t\le x) =  \dfrac1\mu\int_0^\infty \tilde g_x(s)ds.
\end{equation}

Finally, it is not hard to see that the right-hand side of
\eqref{fitrep} is strictly increasing in $v$, and the
right-hand side of \eqref{agerep} is strictly increasing in
$x$, so the limit distributions are nondegenerate.

\section{Proof of Theorem ~4.}
We start with the case $r=1$. 
In this case, conditional on $X_t=k$, the set of
fitnesses has the same law as that of $k$ uniform $(0,1)$
random variables, and hence
\begin{equation}\label{indepen}
P(\f_t^1 \le u|X_t =k) = u^k, \quad 0<u<1.
\end{equation}
This is because (i) the sequence of uniforms created when
$X_t$ jumps is independent of $X_t$, (ii) when $r=1$, the
type that is killed is independent of the types that are
present, and (iii) $k$ uniforms chosen randomly from $n\geq k$ iid
uniforms has the law of $k$ iid uniforms.  For $\lambda\ge
1$, $P(X_t\le K)\to 0$ as $t\to\infty$ for any
$K<\infty$. Applying \eqref{indepen} we obtain
\begin{equation}\label{convto1}
\f_t^1 \to_p 1 \text{ as }t\to\infty .
\end{equation}

To handle $\f^r_t$ for $0<r<1$ we argue that $\f^r_t$ is
stochastically larger than $\f^1_t$. To do this we will use
a coupling that is based on the following definition and
elementary lemma. For positive integers $k$ and sets
$A,B\subset(0,1)$ such that $|A|=|B|=k$, write $A\preceq B$
if $A$ has elements $a_1<\cdots<a_k$ and $B$ has elements
$b_1<\cdots<b_k$ and
\begin{equation}\label{aborder}
a_i\le b_i \text{ for }1\le i\le k.
\end{equation}

\begin{lem}\label{lem:order} Let $A,B\subset(0,1)$ each have $k$
  elements, and suppose $A\preceq B$. Then $A'\preceq B'$ in each of the
  two cases:
\begin{enumerate}[(a)]
\item $A'=A\cup\{w\}$ and $B'=B\cup\{w\}$, where $w\in
  (0,1)$ and $w\notin A\cup B$. 
\item $k\ge 2$, $A'$ is obtained by deleting any
  element of $A$ and $B'$ is obtained by deleting the
  smallest element of $B$. 
\end{enumerate}
In particular, $\max(B')\ge \max(A')$. 
\end{lem}


\begin{proof} For (a), put $a_0=b_0=0$ and
  $a_{k+1}=b_{k+1}=1$. Then for some $0\le i\le k$ and $0\le
  j\le k$, $w\in(a_i,a_{i+1})\cap(b_j,b_{j+1})$, where
  necessarily $j\le i$. Then
\[
a'_\ell=\begin{cases}
a_\ell &\text{if } \ell \le i,\\
w &\text{if } \ell = i+1,\\
a_{\ell-1} &\text{if } \ell \ge i+2,
\end{cases}
\qquad
b'_\ell=\begin{cases}
b_\ell &\text{if } \ell \le j,\\
w &\text{if } \ell = j+1,\\
a_{\ell-1} &\text{if } \ell \ge j+2.
\end{cases}
\] 
It is easy to check that $a'_\ell\le b'_\ell$ for all
$\ell$.

For (b), if $a_i$ is the element deleted from $A$, then
$a'_\ell=a_i$ if $\ell<i$ and $a'_\ell=a_{\ell+1}$ if
$\ell>i$, while $b'_\ell=b_{\ell+1}$ for $\ell\ge 2$. Again,
it is easy to check that $a'_\ell\le b'_\ell$ for each $\ell$.
\end{proof}

Fix $0<r<1$. To be very clear about the coupling we need we
note that our system can be constructed from (i) the
birth-death process $X_t,t\ge 0$, (ii) an iid sequence of
uniform $(0,1)$ random variables $v_n,n\ge 0$, (iii) a
sequence of iid mean $r$ Bernoulli random variables
$\vep_n,n\ge 0$, and (iv) independent random variables
$W^n_k, n,k\ge 1$, $P(W^n_k=j)=1/k$ for $1\le j\le k$. When
$X_t$ makes its $n$th transition up the uniform variable
$v_n$ is added to the current set of types. If $X_t$ makes
it's $n$th transition down, and there are $k$ types before
the transition, the least fit type is deleted if $\vep_n=0$
while if $\vep_n=1$ and $W^n_k=j$ then the $j$th largest
type is deleted. This gives a construction of 
a set of types at time
$t$,$F^r(t)=\{f^r_1(t),\dots,f^r_{X(t)}(t)\}$, with 
$\f^r_t=\max(F^r(t))$. 

Using the same collection of variables we may construct a
second set of types $F^1(t)=\{f^1_1(t),\dots,
f^1_{X(t)}(t)\}$ as follows. Put $F^1(0)=F^r(0)$, so
certainly $F^1(0)\preceq F^r(0)$. Now suppose $F^1(t)\preceq
F^r(t)$ and the elements of each set are put in increasing
order.  If a jump up occurs for the birth process, and $w$
is the value of the uniform random variable added to
$F^r(t)$ is is also added to $F^1(t)$, preserving the
$\preceq$ relationship by Lemma~\ref{lem:order}. If a jump
down occurs, and the appropriate $\vep_n=1$ and $W^n_k=j$,
the $j$th largest element of each set is deleted. If
$\vep_n=0$, the $j$ largest element of $F^1(t)$ is still
deleted, while the smallest element of $F^r(t)$ is
deleted. Again by Lemma~\ref{lem:order}, the $\preceq$
relationship is preserved. Furthermore, this gives a
construction of the fitness process when $r=1$, i.e., the law
of $\max\{F^1(t)\}, t\ge 0 $ is the same as that of
$\f^1_t,t\ge 0$.

This gives a construction with $\phi^r(t)\geq \phi^1(t), t\ge
0$. In view of \eqref{convto1} this proves $\phi^r_t\to_p1$
as $t\to\infty$.

\end{document}